\documentclass[12pt]{article}
\usepackage{amsmath,amssymb,amsfonts,amsthm,url,xspace,graphicx,enumerate}
\newtheorem{theorem}{Theorem}

\theoremstyle{remark}

\newcommand{\Z}{\mathbb{Z}}
\newcommand{\R}{\mathbb{R}}

\renewcommand{\H}{{\mathbb H}}
\newcommand{\eps}{\varepsilon}
\newcommand{\be}{\begin{equation}}
\newcommand{\ee}{\end{equation}}
\newcommand{\old}[1]{}

\newcommand{\isom}{{\text{Isom}(\H)}}
\newcommand{\iid}{i.i.d.\,}

\usepackage[usenames,dvipsnames]{color}

\begin{document}
\title{Right-angled hexagon tilings of the hyperbolic plane}
\author{R. Kenyon\thanks{Department of Mathematics, Brown University, Providence, RI 02912; research supported by the NSF grant DMS-1208191 and the Simons Foundation}}
\date{}
\maketitle

\section{Introduction}

In the spectrum of random planar structures studied in mathematics,
the most basic examples are lattice-based systems like random tilings of the Euclidean plane with polyominoes;
in these the randomness is encoded in the combinatorial arrangement of tiles. A similar basic setting
is in classical lattice statistical mechanics, where the randomness is encoded in the states of ``particles"
sitting at lattice sites (typically with a finite number of states), and interacting with nearest neighbors.

At the other end of the spectrum of random planar processes, one can consider random metrics on $\R^2$,
as in the case of Liouville quantum gravity,
where there are no local interactions or combinatorics.  

Between these two extremes of either pure combinatorics and fixed geometry, or no combinatorics and random geometry,
is the case of non-lattice statistical mechanics, 
in which particles are free to change position but still interact with each
other. Very few examples of this type have been successfully studied.

We discuss here a case of this intermediate type,
where the randomness is encoded in the spatial positions of interacting components.

We work in the hyperbolic plane $\H$, and our components are right-angled hexagons (RAHs)
of varying shape. The components ``interact" in such a way as to form global, edge-to-edge tilings
of the hyperbolic plane. One can take a different point of view on the same system
and consider the individual components to be
bi-infinite geodesics, which interact with each other so as to intersect orthogonally (and so that the complementary
components are hexagons). 

Other models of random geometric structures in $\H^2$ have been studied, see e.g. \cite{CW}. However in these
cases the underlying measures are simpler in a Markovian sense: once a geodesic is determined
the left and right half-space structures are independent. Although our model eventually boils down to a similar argument,
the analogous condition is a priori much less evident.

There is a periodic tiling with regular (all sides of equal length) RAHs, see Figure \ref{regulartiling}.
It is not at all clear that this structure is flexible in a way that preserve angles.

\begin{figure}[htbp]
\centering
\includegraphics[width=3in]{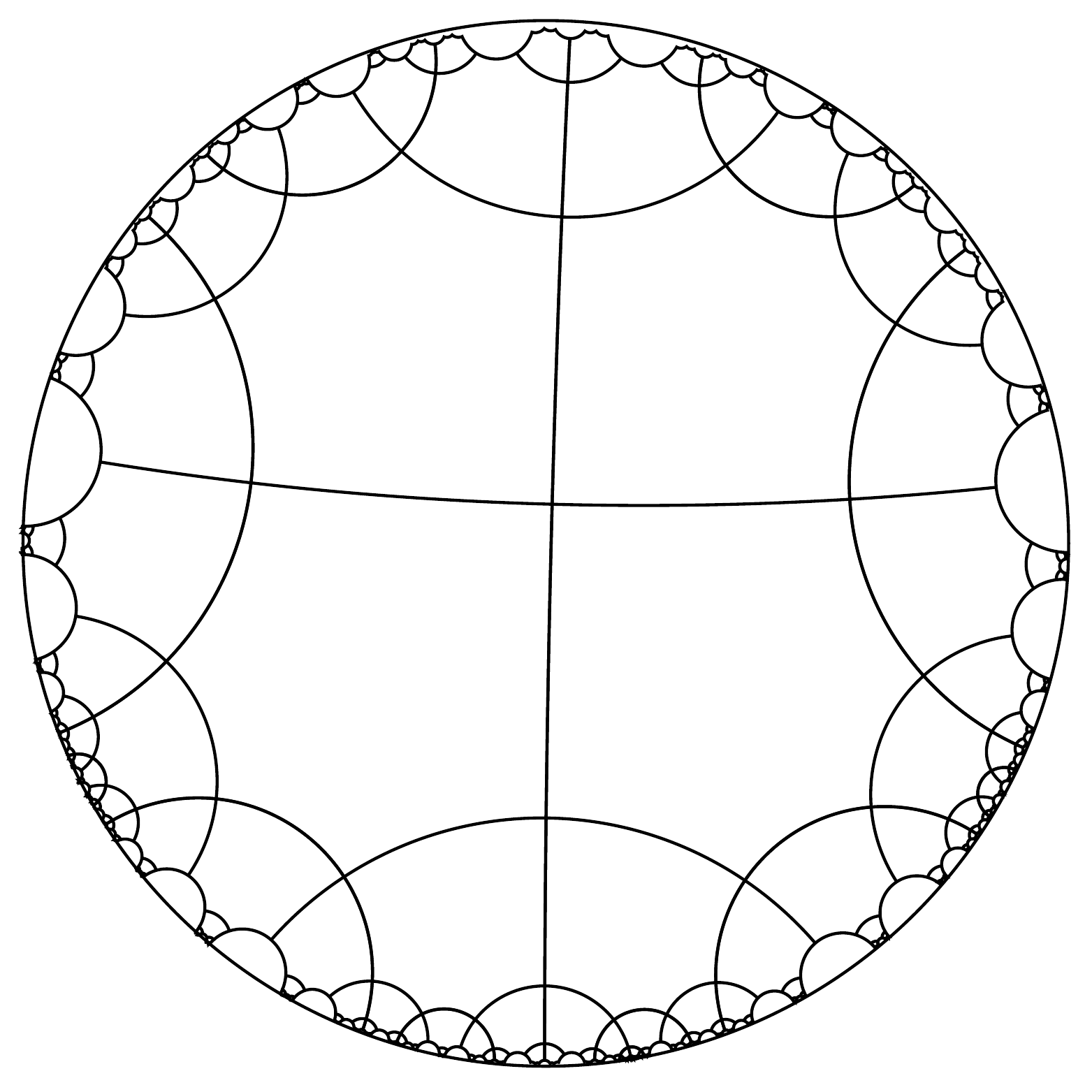} 
\caption{The tiling with regular RAHs.}
\label{regulartiling}
\end{figure}

Let $0\in\H$ denote the center of the Poincar\'e disk. 
Let $\Omega$ be the space of edge-to-edge tilings of the hyperbolic plane $\H$ with RAHs. It has a natural topology: two tilings $T_1,T_2$ are close
if there is a homeomorphism $\phi:\H\to\H$ which is $\eps$-close to the identity ($|\phi(z)-z|<\eps$)
on a large neighborhood of $0$.  The space $\Omega$ also comes with a natural
action of the isometry group of $\H$; it is thus a Riemann surface lamination, whose 
leaves are generically hyperbolic planes.
 
One way to explore $\Omega$ is to study invariant measures on it. Given a measure $\mu$ on RAHs one
can try to construct a measure on $\Omega$ for which the marginal (the induced measure on tiles) is 
given by $\mu$. We prove this for a family of measures $\mu$:

\begin{theorem} Given a probability measure $\nu$ on $\R_+$ of full support, let $\mu$ be the measure on RAHs 
whose edge lengths $\ell_1,\ell_3,\ell_5$ are independent and distributed according to $\nu$. 
Then there is an $\isom$-invariant probability measure on $\Omega$ of full support in which the tile containing the origin (and therefore every tile) has measure $\mu$. 
\end{theorem}

An analogous result holds for tilings with right-angled pentagons (Section \ref{RAPs}) and presumably
for $k$-gons with $k>6$ as well. For shapes with fewer sides, such as quadrilaterals with angles $2\pi/5$, say,
we have been unable to formulate a similar result, probably because the dimension of the space of shapes is too small. 

One way to state the invariance under $\isom$ of a measure on $\Omega$ is as follows.
Take a $\mu$-random tiling. Pick an isometry $\phi$ of $\H$ independent of the tiling. Applying $\phi$ to the tiling results in a
new $\mu$-random tiling. In fact, since all tiles have the same area, the random tiling ``looks the same around any tile'', in the following sense.
Choose from the tiling any tile in a combinatorial manner, that is,
which does not depend on the geometry of the current tiling, except for the information about which tile contains the origin.
Translate this tile so that the origin $0\in\H$ is at a uniform random location in it, and then choose a random
rotation fixing the origin. Then the result is
another exact sample of the measure. 

There are leaves of $\Omega$ which are closed Riemann surfaces, corresponding to RAH tilings which are periodic
under the action of a cocompact subgroup of isometries. These leaves come in families whose dimension is
(as one can show) half the dimension of the Teichmuller space of the associated Riemann surface. 
Each such a family supports many isometry-invariant measures, supported on group-invariant RAH-tilings. 
We will not address here the question of what natural 
measures are supported on these subspaces; we are concerned rather 
with measures of \emph{full support} in $\Omega$.

Another motivation for studying the space $\Omega$ comes from integrable systems. There is a surprising
connection between the geometry of an RAH and the Yang-Baxter equation for the Ising model. The Yang-Baxter
equation, or star-triangle equation, for the Ising model is a local rearrangement of a graph preserving
the so-called partition function of the model. 
The Yang-Baxter equation is a hallmark of an underlying integrable system, and for the Ising model
one which has been studied in \cite{KP}. 
One can translate the Ising integrable structure into an integrable structure on certain spaces of surfaces tiled by RAHs.
The measures we discuss below
are relevant to this system: the ``trigonometric" configurations here are in fact fixed points of the system, and the measure 
$\mu_0$ discussed below is an invariant measure.  This will be the subject of a forthcoming work. 

\old{
One can (conjecturally) translate the Ising integrable structure into an integrable structure on the space $\Omega$.
There is in fact a proof of existence of this integrable structure on 
many of the families parametrizing fixed-genus surfaces mentioned in the previous paragraph. The measures we discuss below
are relevant to this system: the ``trigonometric" configurations are in fact fixed points of the system, and the measure 
$\mu_0$ is an invariant measure.  This will be the subject of a forthcoming work. }
\medskip

\noindent{\bf Acknowledgements.} We thank Oded Schramm and Andrei Okounkov for helpful discussions,
and the referee for several suggestions for improvments.

\section{Measures on RAHs}
\subsection{Hexagons}
The side lengths $(\ell_i)_{i=1,\dots,6}$ of a right-angled hexagon satisfy the following relation.
Let $(a,B,c,A,b,C)=(e^{\ell_1},\dots,e^{\ell_6})$. If $a,b,c$ are known (but arbitrary in $(0,\infty)$)
then the remaining three are determined by
\begin{eqnarray}\frac{A+1}{A-1} &=& \sqrt\frac{(1+abc)(a+bc)}{(b+ac)(c+ab)}\nonumber\\
\frac{B+1}{B-1} &=&\sqrt\frac{(1+abc)(b+ac)}{(a+bc)(c+ab)}\label{ABC}\\
\frac{C+1}{C-1} &=& \sqrt\frac{(1+abc)(c+ab)}{(a+bc)(b+ac)}.\nonumber
\end{eqnarray}
It is straightforward to see that the map $\Psi\colon(a,b,c)\mapsto(A,B,C)$ is an involution.
In particular the family of RAHs is homeomorphic to $\R_+^3$, parametrized by three non-adjacent edge lengths.

One particularly nice family of hexagons are the \emph{trigonometric} ones,
where 
$$1-a-b-c-ab-ac-ab+abc=0,$$
or equivalently 
$$a=\tan\alpha,~~b=\tan\beta,~~c=\tan\gamma,$$
and $\alpha+\beta+\gamma=\frac{\pi}4.$
In this case opposite sides have equal lengths ($A=a,~B=b,~C=c$, as one can check from (\ref{ABC})).
One can define a measure on trigonometric RAHs by choosing, for example,  $\alpha,\beta,\gamma$
uniformly with respect to Lebesgue measure on the simplex $\{\alpha+\beta+\gamma=\frac{\pi}4\}.$

A more interesting measure $\mu_0$, of full support in $\R^3_+$, 
is given by choosing $\ell_1,\ell_3,\ell_5$ independently and each distributed
with respect to the measure $\nu_0$ on $(0,\infty)$ with density
\be\label{nudef}
F(\ell)\,d\ell = \frac{C\,d\ell}{\sinh^{1/3}(\ell)}
\ee
where $$C=\frac{\sqrt{3} \Gamma(\frac{2}{3}) \Gamma(\frac{5}{6})}{2 \pi ^{3/2}}.$$
In terms of $u=e^\ell\in(1,\infty)$ the density is 
\be\label{udef}f(u)\,du= \frac{2^{1/3}C\,du}{u^{2/3}(u^2-1)^{1/3}}.\ee

The truly remarkable property of this measure is 
\begin{theorem}\label{rahmsr} If $\ell_1,\ell_3,\ell_5$ are chosen i.i.d. with distribution $\nu_0$ then $\ell_2,\ell_4,\ell_6$ will be i.i.d. 
with distribution $\nu_0$ as well.
\end{theorem}

\begin{proof} This is a computation: one needs to show that 
$$f(a)f(b)f(c)\,da\,db\,dc=f(A)f(B)f(C)\,dA\,dB\,dC$$ that is,
that the Jacobian of the mapping from $(a,b,c)$ to $(A,B,C)$ defined by (\ref{ABC})
is 
$$\left(\frac{\partial_{A,B,C}}{\partial_{a,b,c}}\right) = \frac{f(a)f(b)f(c)}{f(A)f(B)f(C)}.$$
\end{proof}

One can ``discover" the distribution $\mu_0$ as follows.
If we assume $a,b,c$ are i.i.d.\! with some density $g(x)dx$ then $A,B,C$ will be distributed
according to the density
$$\Psi_*(g(a)g(b)g(c)da\,db\,dc)=g(a)g(b)g(c)\left(\frac{\partial_{a,b,c}}{\partial_{A,B,C}}\right)dA\,dB\,dC.$$
Massaging the right-hand side of 
this expression into a form separating out the $A,B$ and $C$ dependence, one sees that
there is a unique choice of $g$ (up to scale), given by (\ref{udef}), for which it can be written as $g(A)g(B)g(C)dA\,dB\,dC$.
It follows that $\mu_0$ is the only probability measure with this ``independence-preserving" property.
\medskip

\noindent{\bf Question.} Does this probability measure $\mu_0$ have a geometric significance? 
\medskip

One can of course choose any other measure $\nu$ on $\R_+$ and define a corresponding measure 
$\mu=\mu(\nu)$ on RAHs by choosing lengths $\ell_1,\ell_3,\ell_5$ independent and $\nu$-distributed.
The remaining lengths  $\ell_2,\ell_4,\ell_6$ are then determined (and will have the same 
marginal distribution as each other but will not be independent and typically not
distributed according to $\nu$). The resulting
measure $\mu=\mu(\nu)$ on RAHs is invariant under rotation by two indices. 

\section{Tiling}
\subsection{Hextrees}
We will build an $\isom$-invariant probability measure $\lambda$ on $\Omega$ with the property that the 
marginal distribution of a tile (the probability measure restricted to any single tile)
is $\mu(\nu)$ for some $\nu$ as above. 
The building blocks of this measure are collections of RAHs glued together in a ternary-tree structure
as in Figure \ref{hextree}. We call such sets \emph{hextrees}.
\begin{figure}[htbp]
\centering
\includegraphics[width=3in]{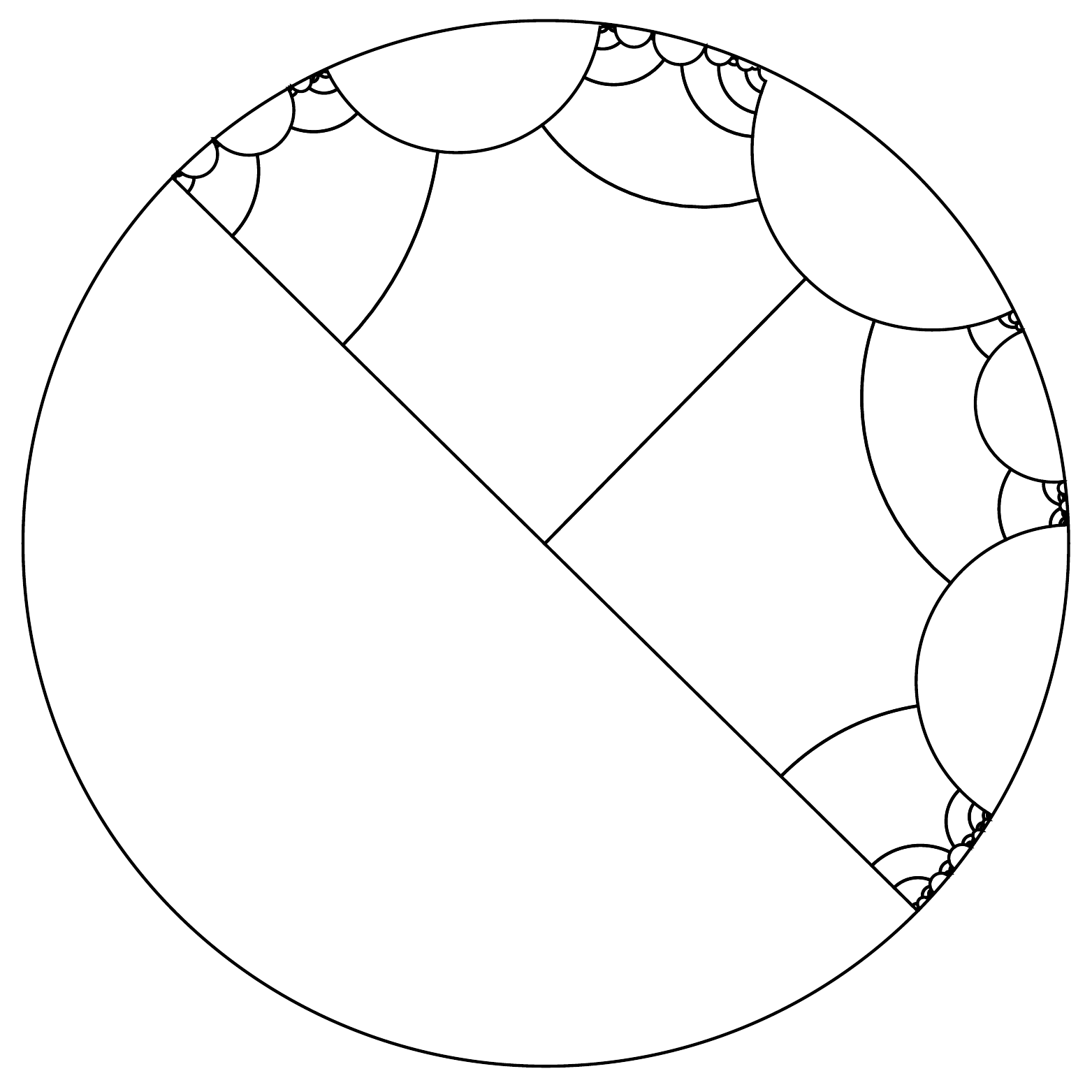} 
\caption{A random hextree in $\H$. For the purposes of illustration this one was constructed using a measure on side lengths which is more tame than $\nu_0$.}
\label{hextree}
\end{figure}
That is, for each edge of the infinite regular degree-$3$ tree take an independent random real
variable distributed according to $\nu$.
Then, for each vertex of the tree, use the three adjacent random variables as lengths of edges $1,3,5$ of a RAH. These RAHs
are glued together so that RAHs at adjacent vertices are glued along the corresponding edge. 
We refer to these glued edges as the cross-edges of the hextree, since they cross between boundary components.

\subsection{Gluing}
Now the construction of the tiling in $\Omega$ is as follows. Start with a hextree as above; 
for each of its boundary geodesics construct another
hextree, conditioned on having a boundary with the same subdivision points as the first, but otherwise independent of the first 
(we discuss how this is accomplished below); then glue these
hextrees along their common boundary so that the individual RAHs meet edge-to-edge. 
We continue gluing hextrees to existing boundaries in this way, always choosing independent new hextrees conditioned to agree along the
common boundary. 

To show that the result is a tiling in $\Omega$, it remains to show all of $\H$ is covered.
This is proved in section \ref{coveringsection} below.

\begin{figure}[htbp]
\centering
\includegraphics[width=3in]{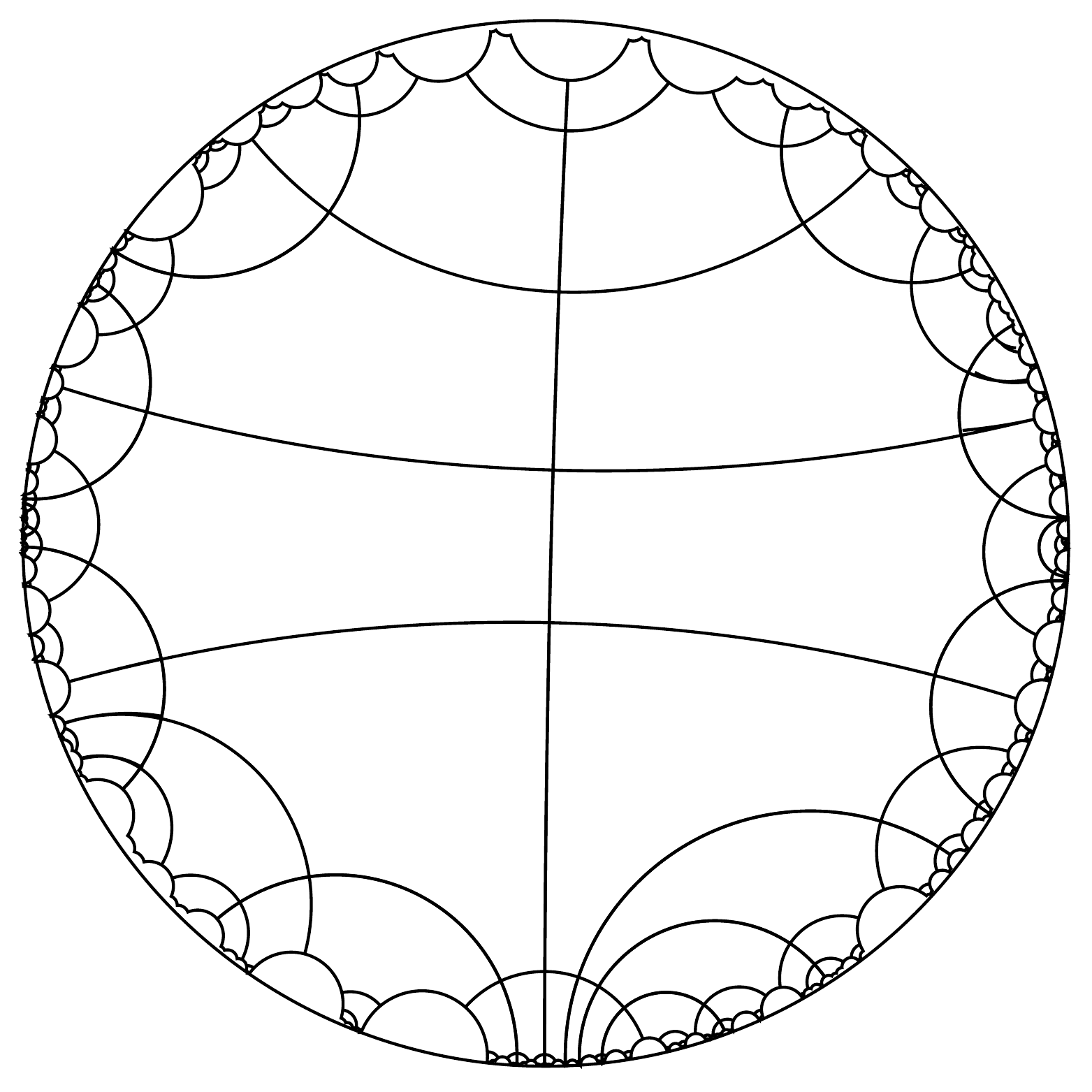} 
\caption{A RAH tiling. For this figure we used the measure $\nu$ on edge lengths 
with density (for $u=e^\ell$) $f(u)=C(u-1)e^{a_0(1-u)}du$, with mean $4$.}
\label{RAHtiling}
\end{figure}

How does one construct a hextree with the same boundary data (along one boundary geodesic) as a given hextree?
Given a hextree $R$ let $Y$ be the boundary data
(the sequence of edge lengths) along one boundary of $R$, and $X$ be the remaining lengths defining the rest of the hextree.

There is a general theorem which applies in this situation \cite{Kal}:  given a joint random variable $(X,Y)\in\Omega_1\times\Omega_2$
(that is, a random variable in a product space), there is a conditional probability measure $(X,Y)|Y$ obtained by conditioning
on $Y$; a random sample $(X',Y)|Y$ of this measure can be chosen independently of $(X,Y)$ conditional on $Y$. 
So to glue a hextree to another
we just resample this conditional measure.
In practice this resampling is complicated; see section \ref{sampling} below. 

The M\"obius invariance of the resulting measure is a consequence of the reversibility of the gluing procedure: each new hextree is distributed
according to the same measure $\mu$, and the joint probability measure on the glued object is symmetric: 
one cannot tell where the construction began and in what order the hextrees were glued. 
Furthermore, within each hextree each RAH has the same distribution, with edges $1,3,5$ which are
independent, so there is no information about which RAH was the ``first". In particular,
as mentioned earlier, if we choose any RAH in the tiling 
translate it so that the origin $0\in\H$ is at a uniform random location in that RAH, the result is
another exact sample of $\lambda$.
\medskip

\noindent{\bf Question.} Is the measure on tilings constructed from the above measure $\mu_0$ characterized by some natural
property like entropy maximization?
\medskip

\noindent{\bf Question.} 
Is there an $\isom$-invariant measure on $\Omega$ supported on trigonometric RAHs?

\subsection{Covering property}\label{coveringsection}
To show that the union of hextrees covers all of 
$\H$, we need to show that hextrees are typically ``thick'': if some
cross-edge of a hextree is short then along the geodesic containing this
short cross-edge the other cross-edges of nearby hextrees are typically no shorter. 

This follows from the reversibility of the construction
of the tiling. Let $x_0$ be a cross-edge. Choose one of the two cross-edges adjacent to it on the same geodesic.
This cross-edge $x_1$ is just as likely
to be greater than $x_0$ as smaller than $x_0$, by reversibility. Similarly for the next cross-edge
$x_2$, and so on. So the chance of getting a sequence of smaller and smaller
cross-edges starting from $x_0$ goes to zero with the length of this sequence. 

Thus along the ``cross-edge geodesics" there are no accumulation points of vertices of RAHs,
and so all of $\H$ is covered.

\subsection{Resampling}\label{sampling}

We are given a hextree $R$ with boundary data $Y=\{\dots,y_{-1},y_0,y_1,\dots\}$, and wish to construct another, independent of $R$
conditional on having the same boundary data. It suffices to construct only the layer of RAHs adjacent to the boundary,
since the remaining RAHs can be chosen by choosing their cross-edge lengths independently of everything 
as in the original definition of $\mu$. 
To construct the RAHs adjacent to the boundary $Y$, it suffices to construct the sequence $x_i$ of
edge lengths of the edges incident to $Y$ 
(ending on $Y$ and perpendicular to it, so that $x_i$ ends between $y_i$ and $y_{i+1}$)
since the three consecutive edges $x_{i-1},y_i,x_i$ will determine the RAH at position $i$. 

We construct the $x_i$ using a successive approximation procedure, 
starting initially from the corresponding lengths $x_i'$ of $R$,
that is, from the edge lengths of the edges of $R$ incident to $Y$ from the other side. 

We use the fact that the sequence $\{(x_i,y_i)\}$ is Markovian: given one pair $(x_i,y_i)$ the values $(x_j,y_j)$ for $j>i$ are independent
of the values $(x_j,y_j)$ for $j<i$. This follows from the construction of $\mu$ from 
the sequence of independent cross-edge lengths:
the values $(x_j,y_j)$ for $j>i$ depend only on the choices of cross-edge lengths on one side of $x_i$,
the values $(x_j,y_j)$ for $j<i$ depend only on the cross-edge lengths on the other side of $x_i$.

From the sequence $\{x'_i\}_{i\in\Z}$, erase all $x'_{2i}$ for all $i$. Then for each $i$, resample $x'_{2i}$ according to the conditional
measure defined by $x'_{2i-1}, y_{2i-1},y_{2i},x'_{2i+1}$. By this we mean, by the Markov property, the marginal distribution of $x'_{2i}$
(given everything else) only depends on these four values. Resampling each $x'_{2i}$  is then a straightforward finite-dimensional computation.

Now repeat for the odd indices: erase each $x'_{2i+1}$ and resample it according to its four neighboring values. 
Now iterate, resampling successively the even values then the odd values over and over. The resampling map
preserves the measure $\mu$ conditioned on $Y$, and the $x_i$'s will quickly become decorrelated from their initial values,
that is, after a small number of iterations the state is independent of the initial state. 
To be a little more precise, if we are interested in functions of $\{x_i\}$
which depend only on indices in the interval $[-m,m]$, then after a polynomial number (in $m$) 
of iterations the system is close to uncorrelated.
These facts have been established in \cite{BRS} in a similar setting.

Although it should be possible to show that the dependence of $x_i$ on $y_j$ decays exponentially
with $|i-j|$, we have not attempted to prove this here.

\section{Pentagons} \label{RAPs}
Let $\ell_1,\dots,\ell_5$ be the side lengths of a right-angled pentagon (RAP).
Any two of these determine the rest: we have for example
$$\cosh\ell_4 = \sinh\ell_1\sinh\ell_2,$$
and its cyclic rotations. Explicitly, given $\ell_1,\ell_2$, which must satisfy $\sinh\ell_1\sinh\ell_2>1$, we have
\begin{eqnarray*}
\sinh\ell_3 &=& \frac{\cosh\ell_1}{\sqrt{\sinh^2\ell_1\sinh^2\ell_2-1}}\\
\sinh\ell_4 &=& \sqrt{\sinh^2\ell_1\sinh^2\ell_2-1}\\
\sinh\ell_5 &=& \frac{\cosh\ell_2}{\sqrt{\sinh^2\ell_1\sinh^2\ell_2-1}}.
\end{eqnarray*}
 
There is a natural measure $\tau$ on right-angled pentagons (RAPs) which is invariant under rotations,
defined as follows (one can discover this measure in a similar manner as one discovers
the measure $\nu_0$ as discussed after the proof of Theorem \ref{rahmsr}). If we parametrize by $\ell_1,\ell_2$, the density is
$$f(\ell_1,\ell_2)\, d\ell_1\, d\ell_2 = \frac{C\,d\ell_1\, d\ell_2}{\sqrt{\sinh^2\ell_1\sinh^2\ell_2-1}}=\frac{C\,d\ell_1\, d\ell_2}{\sinh\ell_4}$$
where $C=K(1/\sqrt{2})^2$, $K$ being the complete elliptic integral of the first kind. For this
measure each edge has expected length $\pi/2$. 

In terms of exponential of the lengths we have (with $a=e^{\ell_1},$ and  $b=e^{\ell_2}$)
$$f(a,b)\,da\,db= \frac{\frac14C\,da\, db}{\sqrt{(ab+a+b-1)(ab-a+b+1)(ab+a-b+1)(ab-a-b-1)}}.$$

As in the case of right-angled hexagons, one can construct an $\isom$-invariant measure on tilings 
of $\H$ with RAPs whose marginal distribution is $\tau$, as follows. 

Gluing four RAPs edge-to-edge around a vertex produces a right-angled octagon after merging four pairs of edges.
The lengths of its four edge lengths $\ell_1,\ell_3,\ell_5,\ell_7$ opposite to the central vertex (which were originally edges of pentagons)
satisfy one relation:
\be\label{cccc}\cosh\ell_1\cosh\ell_5 = \cosh\ell_3\cosh\ell_7.\ee
Given four lengths satisfying this relation, there is a one-parameter
family of choices for the central four pentagon edges. 

We can glue such octagons together along the $\ell_{i}$ for $i$ odd, choosing lengths at each glued edge which are independent
subject to (\ref{cccc})---as explained below---to make a thickened
$4$-valent tree, an ``octatree". Then we can proceed as above, gluing octatrees to make a tiling of $\H$.

To choose the edge lengths of the cross-edges of the $4$-valent tree, subject to (\ref{cccc}) at each vertex, proceed as follows.
Take a bi-infinite straight path (not turning left or right at a vertex) in the tree through a fixed vertex $v$; 
assign values to its edges which are \iid and distributed
according to the $\cosh$ of the edge lengths. Assign the 
edges along the other straight path through $v$ also \iid except that the values on the two edges adjacent to $v$ must have product equal
to the product of values on the other two edges adjacent to $v$. Now proceed outwardly from $v$;
along each new bi-infinite geodesic encountered, assign values on it \iid subject to the condition (\ref{cccc}) at each vertex.
Since this is a tree the reversibility of the assignments guarantees that the resulting object is $\isom$-invariant.

\old{
we first orient the edges of the $4$-valent tree so that each vertex has exactly one outgoing edge. There are many ways to 
do this, in fact they are in bijection with the ends of the tree (just direct each edge towards the end).
Now choose edge lengths recursively: given three edge lengths of edges
coming in to a vertex, the fourth length is determined uniquely by (\ref{cccc}). 
Of course one has to start somewhere. Start with a finite tree, consisting of the neighborhood of the origin of radius $R$.
Root this tree at some leaf;
and choose independently the values on the non-root edges, according to the measure $\nu$. Then use (\ref{cccc}) to recursively
define the remaining edge lengths. Now take a limit as $R\to\infty$. Each edge length will be distributed as $\nu$,
and at each vertex $v$, if we consider the four subtrees out of $v$, the values are independent subject to the single
identity (\ref{cccc}) at $v$.
}
%



\begin{thebibliography}{XX}
\bibitem{BRS}
M. Balazs, F. Rassoul-Agha, T. Seppalainen,
The random average process and random walk in a space-time random environment in one dimension.  Comm. Math. Phys. 266 (2006) 499-545.

\bibitem{CW}
N. Curien, W. Werner, 
The Markovian hyperbolic triangulation, J.E.M.S. 15, Issue 4 (2013) 1309Ð1341.

\old{
\bibitem{GS}
C. Goodman-Strauss, 
A strongly aperiodic set of tiles in the hyperbolic plane, Inventiones Mathematicae 159 (2005) 119Ð132.}

\bibitem{Kal}
O. Kallenberg, Foundations of modern probability, 2nd ed. (2002) Springer. 

\bibitem{KP}
R. Kenyon, R. Pemantle, Double-dimers, the Ising model and the hexahedron recurrence, arXiv:1308.2998.
\end{thebibliography}
\end{document}